\documentclass{amsart}
\usepackage[utf8]{inputenc}
\usepackage[T1]{fontenc}
\usepackage{amssymb, amsthm}
\usepackage{mathrsfs}
\usepackage{paralist}
\usepackage{mathtools}

\newenvironment{enumb}{\begin{compactenum}[(a)]}{\end{compactenum}}

\newcommand{\oper}[1]{\operatorname{#1}}
\newcommand{\Gal}{\operatorname{Gal}}
\newcommand{\zet}{\mathbb{Z}}
\newcommand{\abs}[1]{\ensuremath{\lvert #1 \rvert}}
\newcommand{\Aut}{\operatorname{Aut}}

\newtheorem{thm}{Theorem}
\newtheorem{dfn}[thm]{Definition}
\newtheorem{rem}[thm]{Remark}
\newtheorem{lem}[thm]{Lemma}
\newtheorem{cor}[thm]{Corollary}
\newtheorem{exm}[thm]{Example}

\newcommand{\Ker}{\operatorname{Ker}}
\newcommand{\klmm}[1]{\ensuremath{\left( #1 \right)}}
\newcommand{\bares}[1]{ \overline{#1\vphantom{#1}} } 
\newcommand{\erz}[1]{\ensuremath{\langle #1 \rangle}}

\newcommand{\oh}{\mathcal{O}}
\newcommand{\pp}{\ensuremath{\mathfrak{p}}}
\newcommand{\cha}{\operatorname{char}}
\newcommand{\en}{\mathbb{N}}
\newcommand{\ord}{\operatorname{ord}}
\newcommand{\pe}{\mathbb{P}}

\newcommand{\rk}{\oper{rk}}

\newcommand{\ptrunc}[2]{\ensuremath{\trunc{\frac{#1}{#2}}}}
\newcommand{\prest}[2]{\ensuremath{\left\{{\frac{#1}{#2}}\right\}}}
\newcommand{\Inj}{\operatorname{Inj}}

\newcommand{\mm}{\mathscr{M}}

\newcommand{\norm}[1]{\oper{N}\klmm{ #1 } }
\newcommand{\trunc}[1]{\ensuremath{\lfloor #1 \rfloor}}
\newcommand{\prodd}{\prod\limits}
\newcommand{\summ}{\sum\limits}
\DeclarePairedDelimiter\ceil{\lceil}{\rceil}
\DeclarePairedDelimiter\floor{\lfloor}{\rfloor}
\begin{document}
\title[Conductor density of local function fields]{The conductor density of local function fields with abelian Galois group}

\author{J\"urgen Kl\"uners}
\address{University Paderborn, Department of Mathematics,
Warburger Str. 100, 33098 Paderborn, Germany} 

\email{klueners@math.uni-paderborn.de}

\author{Raphael M\"uller}
\address{University Paderborn, Department of Mathematics,
Warburger Str. 100, 33098 Paderborn, Germany} 

\email{rmuelle2@math.uni-paderborn.de}


\begin{abstract}
We give an exact formula for the number of $G$-extensions of local function
  fields $\mathbb{F}_q((t))$ for finite abelian groups $G$ up to a conductor
  bound. As an application we give a lower bound for the corresponding counting
  problem by discriminant.
\end{abstract}
\maketitle

\section{Introduction}

Let $F= \mathbb{F}_q((t))$ be the Laurent series field over the finite field with $q=p^f$ elements.
Given a finite abelian group $G$,
we are interested in analysing the function
$$ Z(F,G;n) := \abs{ \{ E/F \text{ normal }: \Gal(E/F) \cong G  \ \text{and} \ \norm{\mathfrak F(E/F) } \leq q^n\} },$$ 
where $\norm{\mathfrak F (E/F)}$
is the  norm of the conductor.

Much progress has been made in analysing the asymptotic behaviour of
similar counting functions, typically weighting by discriminant as in
\cite{Wr} over global fields and 
over global function fields (see \cite{EV}). Both papers
leave the gap when we count abelian
$p$-extensions over global fields in characteristic $p$. For this, there are results on counting global function fields by conductor and local function fields by discriminant (\cite{La1}, \cite{La2}).  In the number field situation M\"aki \cite{Mae} is counting by conductor.

For the local situation (characteristic 0 and $p$), there is Serre's Mass Formula (see \cite[Th. 2]{Se}) which gives a 
strong summation formula over all totally ramified extensions over a fixed local field weighted by discriminant and the size of its automorphism group.
Over $p$-adic fields, there are only finitely many extensions of a given degree. For a  $p$-group $G$ there are infinitely many extensions with that group (\cite[Th. 2.4.3]{Le}, \cite[Th. 7.5.10]{NSW}
) and we can show the following theorem.

\begin{thm}\label{mainthm} 
 Let $G$ be a finite abelian $p$-group with exponent $\exp(G) = p^{e} $ and $ \alpha_p(G) := \sum\limits_{k=1}^{e} \frac{ p-1}{p^k} \rk_{p^k}(G),$ where $\rk_{p^k}(G) = \dim_{\mathbb{F}_p}(G^{p^{k-1}} / G^{p^k}).$  Then there is a $p^e$-periodic function $\delta_G \colon \zet\rightarrow [-\alpha_p(G),0]$ 
 and a function $\varepsilon(G,q,n)$ 
 such that 
$$Z(F,G;n) =   \frac{  \abs{G}}{\abs{\Aut(G)}} q^{n \alpha_p(G)} q^{\delta_G(n)} \varepsilon(G,q,n) \mbox{ and }$$
  \begin{align*}Z(F,G;n) \sim \frac{\abs{G}}{\abs{\Aut(G)}} q^{n \alpha_p(G)} q^{\delta_G(n)} 
.\end{align*}
\end{thm}
Note, that the constants $\delta_G(n)$ and $ \varepsilon(G,q,n)$ are defined in \eqref{def_delta} and \eqref{eq:eps} and $\lim_{n\rightarrow\infty} \varepsilon(G,q,n)=1$. Furthermore we prove that the constants $\alpha_p(G)$ and $\delta_G(n)$ are
additive with respect to direct products, see Remark \ref{delta-G}.

The function $\delta_G$ shows an oscillation depending on the residue of $n$ modulo $\exp(G)$. 
Hence in order to get a convergence result of the form $$ Z(F,G;n) \sim c(F,G) \cdot q^{n \alpha_p(G)},$$  we need to restrict to an arithmetic progression modulo $\exp(G)$. We remark that
there is an oscillation even for the group $G=C_p$, the cyclic group of order  $p$.

For a general abelian finite group $G$ denote by $G_{p'}$ the coprime to $p$-part of $G$. Then $G = G_p \times G_{p'}$ and 
the existence of a solution to the inverse Galois problem only depends on $G_{p'}$, while the  
$p$-part of $G$ determines the asymptotical growth: If $|G|$ is coprime to $p$, there are at most
finitely many extensions with Galois group $G$.
In Theorem~\ref{thm18}
we give an exact formula for arbitrary finite abelian groups.
The main work lies within $p$-part and Theorem \ref{mainthm}. For $\ell \ne p$, the $\ell$-rank of $G$ is bounded by $2$, and only the prime divisors of $q-1$ contribute a possibly non-trivial factor.



Finally, we prove in Theorem \ref{thm:lower} a non-trivial lower bound for the
distribution of abelian local function field extensions counted by
discriminant. We remark that our lower bound coincides with 
 the asymptotic exponent given in Satz~2.1 in \cite{La1}.

\section{Higher Unit Groups}

Let $\mathbb{F}_q $ be a finite field with $q=p^f$ elements and $F=\mathbb{F}_q((t))$ be the Laurent series ring over $\mathbb{F}_q$. 
Let $\oh_F = \mathbb{F}_q[[t]]$ be the local ring with maximal ideal $\pp=t\oh_F$.
By the main theorem of local class field theory, we get a one-to-one
correspondence of abelian extensions $E/F$
and norm groups $U:= \oper{N}_{E/F}(E^\times )$ in $F^\times$. \\
The conductor exponent $c(U)$ of an open subgroup  $U \leq F^\times$ of finite index is the minimal natural number $n$ such
that $1+\pp^n \leq U$. The conductor of an abelian extension
$E/F$ corresponding to $U$ is $$ \mathfrak F(E/F) = \pp^{c(U)}.$$

\begin{thm}
The mapping 
$
E \mapsto \oper{N}_{E/F}(E^\times) $ defines a bijection between  finite abelian extensions of $F$ and open subgroups of $F^\times$ of finite index.

Moreover the Galois group $\Gal(E/F) $ is  isomorphic to the quotient group $F^\times /U$. 
\end{thm} 
{For a proof, see \cite[Theorem 6.2., p. 154]{FV}.}

Let $G$ be a finite abelian group of exponent $\operatorname{exp}(G)$. 
Recall
$ F^\times \cong \zet \times \mathbb{F}_q^\times \times (1+\pp),$ see Hasse (\cite[Ch. 15]{Ha}).
We define 
$$ U_n := (1+\pp )/ (1+ \pp^{n}) \quad \text{and} \quad X_n  := \zet/ \operatorname{exp}(G)\zet  \times \mathbb{F}_q^\times \times  U_n.$$ 
By class field theory
the counting problem reduces to count the number of open subgroups $U\leq F^\times$ with $F^\times/U$ isomorphic to $G$.
The conductor bound $\norm{\mathfrak F (E/F)}  \leq q^n$ is equivalent to $1+\pp^{n} \leq U$.
Moreover, $F^\times /U\cong G$ implies that $\operatorname{exp}(G)$ annihilates $F^\times /U$.
So for our counting problem it is sufficient 
to consider the subgroups of $F^\times$ containing $$ \operatorname{exp}(G) \zet \times 1 \times \klmm{1+\pp^{n}}$$ 
 which correspond to the subgroups of $X_n$.

By dualising, the number of subgroups of $F^\times$ with quotient isomorphic to $G$ is exactly 
the number of  subgroups of $X_n$ isomorphic to $G$.  Thus we reduce our counting problem 
to counting subgroups in certain finite abelian groups.


In establishing our desired formula, we first study higher unit groups, 
and consider formulas on subgroups of finite abelian groups depending on the $p^k$-ranks of the groups. 

\begin{dfn}
Let $G$ be a finite abelian group and $\ell \in \pe$. Then $$ \oper{rk}_{\ell^k}(G):= \dim_{\mathbb{F}_\ell} ( G^{\ell^{k-1}}/G^{\ell^k})
$$ is the $\ell^k$-rank of $G$. 
If $\ell = p= \cha(F)$,
we use the shorthand notation \[r_k(G) := \oper{rk}_{p^k}(G) \ 
\text{ and set } \  \tilde{r}_k(G) := r_k(G) - r_{k+1}(G) 
.\]
\end{dfn}

Let now $G$ be a  finite abelian $p$-group.
 A sequence of elements $(g_1,\ldots, g_r)$ is called a \emph{group-basis} of $G$ if each element $g\in G$ has a unique
representation $$ g = g_1^{i_1} \cdots g_r^{i_r}, \quad  0\leq i_j < \ord(g_j). $$

If $(g_1,\ldots, g_r)$ is a group-basis of $G$, then $\tilde{r}_k(G)$ is the number of generators with 
$\ord(g_j) = p^k$, i.e., it is the number of cyclic factors of $G$ isomorphic to $C_{p^k}$. 
\begin{lem}
\label{la:Xm-ranks}
Let  $(v_1,\ldots, v_f)$  be an $\mathbb{F}_p$-basis of $\mathbb{F}_q$. Then the following holds:
\begin{enumb}
\item $1+\pp$ has a $\zet_p$-basis 
$$ \{ 1+ v_i t^{k} : k\in \en, p\nmid k, 1 \leq i \leq f\} \quad \text{and} $$
 $$ \{ \bares{1+ v_i t^{k} }: 1 \leq i \leq f, k\leq n-1, p\nmid k \} \quad \text{ is a group-basis  of }U_n.$$
\item For each $v \in \mathbb{F}_q^\times$ and $i \geq 1$ we have in $U_n$ $$ \ord(\bares{1+vt^i} ) = p^{\ceil{\log_p({n/i})}}.$$
\item For all $j \in \en$, $U_n[p^j]$ is generated by $$ \{ 1+v_it^k : 1\leq i \leq f, p\nmid k, \ceil{n/p^j}\leq k \leq n-1 \}.$$
\item For all $k\in \en$ we have $$r_k(U_n) = f\klmm{ \trunc{\frac{n-1}{p^{k-1}}} - \trunc{\frac{n-1}{p^k}}} \quad \text{ and } \quad r_i(X_n) = r_i(U_n)+1 \ \text{for } i=1,\ldots,e.$$ 
\end{enumb}
\end{lem}
\begin{proof}
\begin{enumb}
\item 
\cite[p. 227]{Ha}. 
\item  $U_n$ is a $p$-group of order $q^{n-1}$ as $\erz{1+\pp^i}/\erz{1+\pp^{i+1}} \cong \mathbb{F}_q$ for all $i\geq 1$. 
Let $i\leq n$ and put $\alpha := \bares{1+vt^i } \in \erz{1+\pp}/ \erz{1+\pp^{n}}$ with $v \in \mathbb{F}_q^\times$ and 
$k\in \en$.
Then:
\begin{align*}
\bares{1+vt^i}^{p^k} = 1  & 
 \iff v^{p^k}t^{i p^{k}} \in \pp^{n}  \iff  i p^k \geq n \\
& \iff p^k \geq \frac{n}{i} 
 \overset{k \in \en}{\iff } k \geq \ceil{ \log_p(\frac{n}{i})}.
\end{align*}
\item This is (a) and (b) with $\ceil{\log_p(n/k)} \leq j \iff n/k \leq p^j \iff k \geq n/p^j$.
\item By (a), $\mathscr B = \{ \bares{1+v_jt^i}  : 1\leq i < n \text{ and }p \nmid i\}$ is a group-basis of $X_n$.
Then  $$ r_k(X_n) = \abs{ \{ g \in \mathscr B : \ord(g) \geq p^k \} }.$$
By (b) we have $\ord(\bares{1+v_jt^i}) \geq p^k \iff ip^{k-1} <n \iff ip^{k-1} \leq n-1$, hence 
$$ r_{k}(U_n) = f \cdot \abs{ \{i : i\leq \ptrunc{n-1}{p^{k-1}}, p \nmid i\} }  =  f(\ptrunc{n-1}{p^{k-1}} - \ptrunc{n-1}{p^{k}}).$$

Note that $r_k(X_n) = r_k(U_n)+1$ since $p \nmid \abs{\mathbb{F}_q^\times}$.   \qedhere
\end{enumb} 
\end{proof}

\section{Monomorphisms and Automorphisms of finite abelian groups}

 For $G$ and $ A$ finite abelian groups let $G_p = \{ g \in G : \ord(g) = p^a\ \text{for some } a\in \en \}$   be the $p$-Sylow subgroup of $G$ and let $G_{p^\prime}$ be the coprime to $p$ part of $G$. We define 
 $$ \operatorname{Inj}(G,A) := \{ \phi\ \colon\ G \rightarrow A \ \text{monomorphism}\}, \quad 
 \alpha_G(A) := \abs{\{ U\leq A : U\cong G \} }.$$
 We immediately get
 \begin{equation}\label{lem:injap}
 \alpha_G(A) \cdot \abs{\Aut(G)}= \abs{\Inj(G, A)}.
\end{equation}
  We start with the following reduction to $p$-groups.

 \begin{lem}\label{lem:pstrich} 
 
 $\abs{\oper{Inj}(G,A)  } = \prod\limits_{\ell\in \pe} \abs{\oper{\Inj(G_\ell,A_\ell)}}$ and $\alpha_G(A) = \alpha_{G_p}(A_p) \cdot \alpha_{G_{p^\prime}}(A_{p^\prime}).$
 \end{lem}	
 \begin{proof}
 	Monomorphisms need to preserve the order of elements.
 \end{proof}

%
Thus it is sufficient to consider finite abelian $p$-groups. In the following $G$ and $A$ will be finite abelian $p$-groups with $\exp(G) = p^e$.
As in \cite{La1} we define 
\begin{equation}
\label{eq:fG-1} 
f_G(t_1,\ldots, t_e)  := \prodd_{k=1}^e t_k^{r_{k+1}(G)} \prodd_{j= r_{k+1}(G)}^{r_k(G)-1} (t_k - p^j) .
\end{equation}


\begin{lem}\label{la:injaut-formula}
 Let $t(A) := (p^{r_1(A)}, \ldots, p^{r_e(A)})$ for an abelian $p$-group $A$. Then:
 \begin{enumb}
   \item $\abs{ \Inj(G,A)} = f_G(t(A)) = \prodd_{k=1}^e p^{r_k(A) r_{k+1}(G)}\prodd_{j= 0}^{\tilde r_k(G) -1}  (p^{r_k(A)} -p^{r_{k+1}(G) +j}) $,
  \item $\abs{\Aut(G)} = \abs{\Inj(G,G)} = f_G(t(G))$.
 \end{enumb}
\end{lem}
The formula goes back to works of Delsarte \cite{De}. A proof can be found in \cite{La1}, Lemma A.1 and Remark A.3., where we use $ \tilde r_k(G) = r_k(G) - r_{k+1}(G)$.

\begin{rem}\label{la:d-function} 
We get another formula which is useful for asymptotic considerations:
\begin{equation}\label{eq:fG-2} \abs{\Inj(G,A)}  = \prodd_{k=1}^e p^{r_k(A)  r_k(G)} \prodd_{j= 0}^{\tilde r_k(G) -1} \klmm{ 1- \frac{p^{r_{k+1}(G) + j}}{p^{r_k(A)}}}.
\end{equation}
 \end{rem}
 \begin{proof}
In \eqref{eq:fG-1}, we use $t_k - p^j = t_k(1-\frac{p^j}{t_k})$  and make an index shift to obtain \eqref{eq:fG-2} by plugging in
$t_k=p^{r_k(A)}$.
   \end{proof}

We want to apply these formulas to the norm groups whose $p^k$-ranks involve ceiling operations. 
In the following we denote by $\prest{a}{b}:=\frac{a}{b}-\ptrunc{a}{b}$<1.
Therefore we need the following lemma:

\begin{dfn} 
For a finite abelian $p$-group $G$ of exponent $p^e$ and $n \in \en$ we define 
\begin{equation}\label{la:trunc}
\delta(n,k):= \prest{n}{p^{k}} - \prest{n}{p^{k-1}}
=\ptrunc{n}{p^{k-1}} - \ptrunc{n}{p^k}  - \frac{(p-1)n}{p^k}, \text{ and}
\end{equation}
\begin{equation}\label{def_delta}
\alpha_p(G) := \sum\limits_{k=1}^{e} \frac{ p-1}{p^k} {r}_k(G) \mbox{ and }
\delta_G(n) := -\alpha_p(G) + \summ_{k=1}^e r_k(G) \klmm{  \delta(n-1,k)}.
\end{equation}
\end{dfn}
We immediately see that $\delta(n,k)$ is $p^k$-periodic and therefore $\delta_G(n)$ is $p^e$-periodic.

 \begin{rem}\label{delta-G}
Let $G$ and $H$ be finite abelian $p$-groups of exponent $\leq p^e$. Then:
 \begin{enumb}
 \item $\delta_G(n)= -\alpha_p(G)+ \summ_{k=1}^e \tilde r_k(G) \prest{n-1}{p^k} $ and $\delta_{G\times H}(n) = \delta_G(n) + \delta_H(n)$, 
\item $\alpha_p(G) = \sum\limits_{k=1}^e \tilde r_k(G) \frac{p^k-1}{p^k}$ and $\alpha_p(G\times H) = \alpha_p(G) + \alpha_p(H)$,
 \item $\delta_G(1) = -\alpha_p(G) \leq \delta_G (n) \leq 0 = \delta_G(0)$.
 \end{enumb}
\end{rem} 
\begin{proof} 
	We use an index shift and $r_{e+1}(G)=0$:
$$\sum_{k=1}^e \tilde r_k(G) \frac{p^k-1}{p^k}=\sum_{k=1}^e  (r_k(G)-r_{k+1}(G)) \frac{p^k-1}{p^k}
=\sum_{k=1}^e r_k(G) (\frac{p^k-1}{p^k} - \frac{p^{k-1}-1}{p^{k-1}})
$$
$$=\sum_{k=1}^e r_k(G) \frac{p-1}{p^k}=\alpha_p(G).$$

Similarly, we get:
\begin{align*}
&\sum_{k=1}^e \tilde r_k(G)  \prest{n-1}{p^k} = \sum_{k=1}^e (r_k(G) - r_{k+1}(G))\prest{n-1}{p^k}\\
=&\sum_{k=1}^e  r_k(G) ( \prest{n-1}{p^k}-\prest{n-1}{p^{k-1}})
=\sum_{k=1}^e  r_k(G) \delta(n-1,k) = \alpha_p(G)+\delta_G(n).
\end{align*}

    With $\tilde r_k(G\times H) = \tilde r_k(G) + \tilde r_k(H)$ for $k\geq 1$ this completes (a) and (b).\\ 
Finally  $-\alpha_p(G) =\delta_G(1)\leq -\alpha_p(G) + \summ_{k=1}^e \tilde r_k(G) \prest{n-1}{p^k} = \delta_G(n) \leq 0=\delta_G(0)$.  
\end{proof}

\begin{exm}
	Let $r\in \en$. \begin{enumb}
		\item If $G= (C_p)^r$, then $ \alpha_p(G) = r \cdot \frac{p-1}{p}.$
		\item If $G=C_{p^r}$ is cyclic, then $ \alpha_p(G) = \sum\limits_{k=1}^r \frac{ p-1}{p^k} = \frac{p^r -1}{p^r} .$
	\end{enumb}
	
\end{exm}

\section{Proof of the Main Theorem} 
In this section we prove Theorem \ref{mainthm}.
We prepare some formulas. 

\begin{rem}\label{exp-zerlegung} Let $n\in\en$. Then:
	$$\prod_{k=1}^e p^{r_k(G) r_k(X_n)} = |G| q^{n\alpha_p(G)} q^{\delta_G(n)}.$$
\end{rem}	
	\begin{proof} For all $k= 1,\ldots , e$ we have
	\begin{equation} \label{eq:rkXn}
	r_k(X_n)  \overset{\text{Lem } \ref{la:Xm-ranks}}{=} 1 + f\klmm{ \ptrunc{n-1}{p^{k-1}} - \ptrunc{n-1}{p^k} }
	\overset{ \eqref{la:trunc} }{=}  1+ f \frac{p-1}{p^{k}} (n-1) + f\delta(n-1,k). 
	\end{equation}
	We get:
	\begin{align*}
	\sum_{k=1}^e r_k(G) r_k(X_n) &\overset{\eqref{eq:rkXn}}{=}\sum_{k=1}^e r_k(G) + f \sum_{k=1}^e r_k(G) \frac{p-1}{p^k} (n-1) + f\sum_{k=1}^e r_k(G) \delta(n-1,k)  \\
	&= \log_p(G) + f n \alpha_p(G) + f \delta_G(n).
	\end{align*}
	Note that $q=p^f$.
	\end{proof}


\begin{thm}\  \label{th:main}
 Let $G$ be a finite abelian $p$-group with exponent $\exp(G)=p^{e}$. Let  
 $ \alpha_p(G)$
 and  $\delta_G(n)$ as defined in \eqref{def_delta} where $\delta_G(\cdot)$ is $p^e$-periodic. Let
 $F=\mathbb{F}_q((t))$  and 
   \begin{equation}\label{eq:eps}
    \varepsilon(G,q,n) := \prodd_{k=1}^e \prodd_{j=0}^{\tilde r_k(G)-1} \klmm{ 1- \frac{p^{r_{k+1} (G) + j-1}}{q^{(p-1)(n-1)/p^k+\delta(n-1,k)}} }.
\end{equation}
 Then we have:
 \begin{enumb}
\item 
$Z(F,G;n) =   \frac{  \abs{G}}{\abs{\Aut(G)}} q^{n \alpha_p(G)} q^{\delta_G(n)} \varepsilon(G,q,n).$ 

\item  
$Z(F,G;n) \sim \frac{\abs{G}}{\abs{\Aut(G)}} q^{n \alpha_p(G)} q^{\delta_G(n)}$ and $\lim_{n\to \infty} \varepsilon(G,q,n) = 1.$
\item For  $x_n=n\cdot p^{e}$, i.e. $x_n\equiv  0 \mod  p^e$ we have  
$Z(F,G;x_n) \sim \frac{|G|}{\abs{\Aut(G)}} q^{x_n \alpha_p(G)} .$
\end{enumb}
\end{thm}
\begin{proof} 
\begin{align*} 
  Z(F,G;n) &= \alpha_G(X_n) \overset{ \eqref{lem:injap}}{=} \frac{\abs{\Inj(G,X_n)}}{\abs{\Aut(G)}} \\ 
    & \overset{\eqref{eq:fG-2}}{=} \frac{1}{\abs{\Aut(G)}}\prod_{k=1}^{e} p^{r_k(G) r_k(X_n)} \prod_{j=0}^{\tilde r_k(G) -1} \klmm{ 1- \frac{p^{r_{k+1} (G) + j}}{p^{r_k(X_n)}} }   \\
& 
{\overset{\text{Rem.} \ref{exp-zerlegung}}{=}} \frac{1}{\abs{\Aut(G)}}|G| q^{n\alpha_p(G)} q^{\delta_G(n)} \prod_{k=1}^{e} \prod_{j=0}^{\tilde r_k(G) -1} \klmm{ 1- \frac{p^{r_{k+1} (G) + j}}{p^{r_k(X_n)}} } \\
 &\overset{\eqref{eq:rkXn},\eqref{eq:eps}}{=} \frac{\abs{G} q^{\delta_G(n)} }{\abs{\Aut(G)}} q^{n \alpha_p(G)} \varepsilon(G,q,n) .
\end{align*}

Using  $|\delta(n-1,k)|<1$ we get $\lim\limits_{n\to \infty} \varepsilon(G,q,n) = 1$ for all $k\geq 1$ and we proved (b). 
Using Remark \ref{delta-G} (c) we see $\delta_G(0) = 0$ which gives (c).   \qedhere
\end{proof}



 \begin{exm}\label{ex:Cp}  
 \begin{enumb}
 \item
Let $G= (C_p)^r$ be elementary abelian. Then 
$\alpha_p(G) = r \frac{p-1}{p}$ 
 and $$ \delta_G(n)  = -\alpha_p(G) + r\cdot\prest{n-1}{p} = \begin{cases} 
                                                                                                                   0, & p\mid n  \\
                                                                                                                   r(\prest{n}{p} -1), & p \nmid n.
                                                                                                                   \end{cases}$$  Hence, if $p$ does not divide $n$ we get
$$ {Z(F,G;n)}  = \frac{\abs{G}  } { \abs{\Aut(G)} } q^{ n \alpha_p(G) } q^{r(\prest{n}{p}-1)}\prod_{j=0}^{r-1} 
\klmm{ 1 - \frac{p^{j-1}}{q^{{(p-1)(n-1)/p +\prest{n-1}{p}} }} }.$$
\item Let $G= C_{p^r}$ be cyclic. Then $\alpha_p(G) = \frac{p^r-1}{p^r}$
and 
$$ \delta_G(n) = -\alpha_p(G) + \prest{n-1}{p^r}  = \begin{cases} 
                                                                                   0, & p^r \mid n\\
                                                                                   \prest{n}{p^r} -1, & p^r \nmid n.
                                                                                   \end{cases}$$ 
  Hence, if $p^r$ does not divide $n$ we get 
 $$ Z(F,G;n)  = \frac{\abs{G}  }{ \abs{\Aut(G)} } q^{  n\alpha_p(G) }  q^{\prest{n}{p^r} -1} 
\klmm{ 1 - \frac{p^{-1}}{q^{ (p-1)(n-1)/p^r + \prest{n-1}{p^r}-\prest{n-1}{p^{r-1}}  } }  }.$$
\end{enumb}

 \end{exm}



\section{Arbitrary finite abelian groups}

We now consider an arbitrary abelian group $G$ and fix a prime number $\ell \in \pe$ with $p \neq \ell$. 
We denote by $G_\ell$  the $\ell$-Sylow subgroup of $G$.

The task is to count the number of open subgroups $U\leq F^\times $ such that $ F^\times / U \cong G_\ell.$
Then the extension given by $U$ is at most tamely ramified, so the conductor exponent is $\leq 1$ and as such
$1+\pp \leq U$. 
Hence we can consider $G_\ell$ as a quotient of $\zet \times \mathbb{F}_{q}^\times \cong \zet \times C_{q-1}.$
Obviously, the only possible quotients isomorphic to $\ell$-groups are groups of the form $ C_{\ell^a} \times C_{\ell^b},$
where $a \geq b$ and $\ell^b\mid (q-1)$. In the following remark we only consider situations, where $G$ is a subgroup of $A$ and
we assume that the exponent of $A$ equals the exponent of $G$.

\begin{rem}\label{rem:lgroups}\ \\
Let  $G = C_{\ell^a} \times C_{\ell^b}$ and $A= C_{\ell^a}\times C_{\ell^d}$ with $a\geq d \geq b$.
\begin{enumb}
\item If $a=b$ or $d=0$, then $ \alpha_G(A) = \alpha_G(G)= 1.$
\item If $a > b$ then $$\alpha_G(A) =\begin{cases} 
                   \ell^{d - b}, & a>d, \\
                   (\ell+1) \ell^{d-b-1} , & a = d.
                   \end{cases}$$
\end{enumb}
\end{rem}
\begin{proof} We write $r_i (A) := \rk_{\ell^i}(A) $ and $r_i(G) := \rk_{\ell^i}(G)$ in this proof.
\begin{enumb} \item If $a=b$ or $d=0$, then we get $G=A$ and therefore $\alpha_G(A) = \alpha_G(G) = 1.$
\item By \eqref{lem:injap} and Lemma~\ref{la:injaut-formula}(a) we have
    \begin{align*} \alpha_{G}(A) & = \frac{ \prod\limits_{k=1}^{a}  \ell^{r_{k+1}(G)r_{k}(A)} \prod\limits_{j=0}^{ \tilde r_{k}(G) -1} (\ell^{r_{k}(A)} - \ell^{r_{k+1}(G) +j}  )}{ 
    \prod\limits_{k=1}^{a}  \ell^{r_{k+1}(G)r_{k}(G)} \prod\limits_{j=0}^{ \tilde r_{k}(G) -1} (\ell^{r_{k}(G)} - \ell^{r_{k+1}(G) +j} )}    \\
    &= \frac{(\ell^{r_{b}(A) } - \ell) (\ell^{r_{a}(A) } -1)}{(\ell^2-\ell)(\ell -1)} \prod_{k=1}^{a}  \ell^{r_{k+1}(G)(r_{k}(A) - r_{k}(G))}  .
    \end{align*}   
   As $r_{k}(G) = r_{k}(A)$ for $k\leq b$ we have $\ell^{r_{b}(G)} -\ell = \ell^2-\ell$. Moreover:
    \begin{align*} \prod_{k=1}^{a}  \ell^{r_{k+1}(G)(r_{k}(A) - r_{k}(G))}  
    &= \prod_{k=b+1}^{a-1}   \ell^{1 \cdot (r_{k}(A)-1)}  = \ell^{\min(a-1,d)-b} \mbox{ and }\end{align*}
    \begin{equation*}  \frac{\ell^{r_{a}(A) } -1}{\ell -1}
= \begin{cases} \ell+1, &   r_{a}(A) = 2 \iff d=a \\
1, & r_{a}(A) =1 \iff b \leq d <a.                                \end{cases} \qedhere\end{equation*}
\end{enumb}
\end{proof}

Note in the following theorem that $X_1 = \zet /\exp(G) \zet \times \zet /(q-1)\zet$. We still use
the notation $G_{p'}$ for the prime to $p$-part of $G$.
\begin{thm}\label{thm18}
Let $G$ be a finite abelian group and $F =\mathbb{F}_q((t))$ with $q=p^f$.
\begin{enumb}
\item $G$ is realisable as a Galois group over $F$ if and only if
  $G_\ell^{q-1}$ is cyclic for all prime numbers $\ell\nmid p$.
\item If $G$ is realisable then for all $n\geq 1$ we have
$$ Z(F,G;n) = 
 Z(F,G_p; n) \cdot \prod_{\ell \mid (q-1)} \alpha_{G_\ell}(C_{\abs{G}}\times C_{q-1}) \leq \frac{(q-1)q}{2} Z(F,G_p; n). $$
\end{enumb}
\end{thm}
\begin{proof}
  We use Lemma
  \ref{lem:pstrich}:$$ Z(F,G;n) = \alpha_{G}(X_n) = \prod_{\ell \in
    \pe} \alpha_{G_\ell}(X_n) =\alpha_{G_p}(X_n) \cdot \prod_{p \neq
    \ell \in \pe} \alpha_{G_\ell}(X_1) .$$ For the last equation we
  use that $\ell \neq p$ and the fact that $X_n/X_1$ is a $p$-group.
  If $G$ is realisable we get for $\ell\neq p$ that $G_\ell$ is a
  quotient of $\zet \times \zet/(q-1)\zet$ and therefore
  $G_\ell^{q-1}$ has to be cyclic. Note that for $\ell\nmid p(q-1)$ we get by Remark \ref{rem:lgroups}
  that $\alpha_{G_\ell}(X_1)=1$.
It remains to show the estimate in (b). We have
\[ \prod_{\ell \mid (q-1) } \alpha_{G_\ell}(X_1) \leq \prod_{\ell \mid (q-1)} \ell^{\nu_\ell(q-1)-1} (\ell+1)  = (q-1)\prod_{\ell \mid (q-1)} \frac{\ell+1}{\ell}\] \[ \leq (q-1) \prod_{k=2}^{q-1} \frac{k+1}{k}  = (q-1) \frac{q}{2}.  \qedhere \]
\end{proof}
Note that the last estimate can be easily improved.

\section{Application to count by discriminants}

The asymptotic behaviour weighted by conductor gives  interesting insights to the counting problem weighted by discriminant. Let $G$ be a finite group and 
$$D(F,G;n) = \abs{ \{ E/F  \ : \  \Gal(E/F ) \cong G, \norm{D(E/F)} \leq q^n \} }$$
be the counting function of local function field extensions with Galois group $G$ and bounded discriminant.
We define for abelian $p$-groups $G$ of exponent $p^e$:
$$ \beta_p(G) := \frac{ \alpha_p(G)}{ \rho(G)}, \qquad\text{where } \rho( G) := \sum_{k=0}^{e-1}\frac{1}{p^k} \klmm{ \abs{G^{p^{k}}} - \abs{G^{p^{k+1}}}} .$$
We use the local version of the conductor-discriminant theorem, see  \cite[Thm. 7.15]{Iw}.
\begin{thm}\label{thm:FDF}
  Let $K'/K$ be a finite abelian extension of local fields,
  then $$D(K'/K) = \prod_{\chi} \mathfrak F(\chi),$$ where the
  product is taken over all characters $\chi$ of $\Gal(K'/K)$ and $\mathfrak F (\chi) = q^{c(\Ker(\chi))}$.
\end{thm}

 The idea of the proof of Theorem~\ref{thm:lower} is contained in \cite[Ch. 2]{La2}.
\begin{thm}\label{thm:lower}
Let $F=\mathbb{F}_q((t))$, $G$ be a finite abelian $p$-group and $n \in \en$.
\begin{enumb}
\item Let $E/F$ be a normal extension with Galois group $G$ and $\norm{\mathfrak F (E/F)}=q^n$. Then $$ \norm{ D (E/F)} \leq  \norm{\mathfrak{F}(E/F)}^{\rho(G)} q^{\abs{G}-1}= q^{n \cdot \rho(G) } q^{\abs{G}-1}. $$
\item There exists a constant $\gamma(F,G)>0$ such that 
$$ D(F,G;n) \geq \gamma(F,G) \cdot q^{n\beta_p(G)}.$$
\end{enumb}
 \end{thm}
\begin{proof}
Let $n$ be the conductor exponent and $U$ be the norm group of $E^\times$. Using $G = F^\times / U$, we have for $k=1,\ldots, e:$
 $$M_k := \{ \chi \text{ character of } G : G[p^{k-1}] \leq \Ker(\chi)  \ \land   G[p^{k}] \not\leq \Ker(\chi) \}.$$
By Lemma \ref{la:Xm-ranks}(c), we have $c(\mathfrak F(\chi)) \leq c(U_n[p^{k-1}]) = \ceil{n/p^{k-1}}$ for all $\chi \in M_k$. \\
Then  we have $ \abs{M_k} = \abs{G/G[p^{k-1}]} - \abs{G/G[p^k]} = \abs{ G^{p^{k-1}}} - \abs{ G^{p^k}}$. 
Moreover, $\summ_{k=1}^{e} \abs{M_k}=\abs{G}-1$ and the $M_k$ are disjoint -- we only miss the trivial character which has trivial conductor.
Thus:
\begin{align*} &\norm{{D(E/F)}} \overset{\text{Thm. }\ref{thm:FDF}}{=} \prod_{\chi \text{ char. of } \Gal(E/F)} \norm{\mathfrak F(\chi) }= \prod_{k=0}^{e-1} \prod_{\chi \in M_k } \norm{\mathfrak F (\chi)} \\
\leq & \prod_{k=1}^{e} \prod_{\chi \in M_k }  q^{\ceil{n/p^{k-1}} } = \prod_{k=1}^{e}  q^{\ceil{n/p^{k-1}} \abs{ M_k }} 
= \prod_{k=1}^{e}  q^{\ceil{n/p^{k-1}}(\abs{G^{p^{k-1}}} - \abs{G^{p^{k}}}) } \\ \leq & \prod_{k=1}^{e} q^{  \klmm{{n/p^{k-1}}+1}(\abs{G^{p^{k-1}}} - \abs{G^{p^{k}}}) } 
= q^{ \summ_{k=0}^{e-1} p^{-k} n (\abs{G^{p^{k}}} - \abs{G^{p^{k+1}}})} q^{\abs{G}-1}  =  q^{n \cdot \rho(G)} q^{\abs{G}-1} .
\end{align*}

Hence $ D(F,G; n\rho(G) +\abs{G}-1) \geq Z(F,G; n).$
We set $ \tilde{n} := \floor{(n-\abs{G}+1)/\rho(G)}$.
By Theorem~\ref{th:main} there exists a constant $C >0$ such that $Z(F,G;\tilde n) \geq C q^{\tilde n\alpha_p (G)}$. Hence in total    
\begin{align*} D(F,G;n) &\geq D(F,G; \tilde n \rho(G)+|G|-1) \overset{(a)}{\geq}  Z(F,G; \tilde n) \\
&\geq C \cdot q^{\tilde n\alpha_p(G)}= C q^{\ptrunc{ n- \abs{G}+1}{{\rho(G)}}\alpha_p(G)} \geq C q^{ \frac{n- \abs{G}+1-\rho(G)}{{\rho(G)}}\alpha_p(G) }\\ &=C  q^{n\beta_p(G)}  q^{(-\abs{G} + 1-\rho(G)) \beta_p(G)} = \tilde{C} q^{n\beta_p(G)}. \qedhere\end{align*}
\end{proof}

Finally, using index shifts we can easily show that the constant $\beta_p(G)$ coincides with the asymptotic exponent given in Satz~2.1 in \cite{La1}
(denoted $a_\pp(G)$ in his work). 

   \bibliographystyle{abbrv}
 \bibliography{mybib} 
 
 \end{document}